\documentclass[11pt]{amsart}
\usepackage{latexsym}
\usepackage{amsfonts,amssymb,amsmath,amsthm,amscd}
\usepackage[mathscr]{eucal}
\usepackage{graphicx}

\newcommand{\g}{\mbox{${\mathfrak g}$}}

\newcommand{\kf}{\mbox{${\mathfrak k}$}}

\newcommand{\p}{\mbox{${\mathfrak p}$}}

\newcommand{\I}{\mbox{${\mathbb I}$}}

\newcommand{\PP}{\mbox{${\mathbb P}$}}

\newcommand{\R}{\mbox{${\mathbb R}$}}

\newcommand{\tr}{{\rm tr}}
\newcommand{\ric}{{\rm Ric}}

\newcommand{\vol}{{\rm vol}}

\makeatletter
\def\numberwithin#1#2{\@ifundefined{c@#1}{\@nocnterrr}{%
  \@ifundefined{c@#2}{\@nocnterr}{%
  \@addtoreset{#1}{#2}%
  \toks@\expandafter\expandafter\expandafter{\csname the#1\endcsname}%
  \expandafter\xdef\csname the#1\endcsname
    {\expandafter\noexpand\csname the#2\endcsname
     .\the\toks@}}}}
\makeatother
\numberwithin{equation}{section}

\newtheorem{thm}[equation]{Theorem}
\newtheorem{lemma}[equation]{Lemma}
\newtheorem{prop}[equation]{Proposition}

\newtheorem{ex}[equation]{Example}
\newenvironment{example}{\begin{ex} \em}{\end{ex}}
\newtheorem{rem}[equation]{Remark}
\newenvironment{rmk}{\begin{rem} \em}{\end{rem}}

\begin{document}

\title{A Hamiltonian approach to the cohomogeneity one Ricci soliton
equations}

\author{Alejandro Betancourt de la Parra}
\address{Mathematical Institute, Andrew Wiles Building,
   Oxford University, OX2 6GG, United Kingdom}
\email{betancourtde@maths.ox.ac.uk}
\thanks{A. Betancourt is supported by the Mexican Council of Science and Technology
(CONACyT)}
\author{Andrew S. Dancer}
\address{Jesus College, Oxford University, OX1 3DW, United Kingdom}
\email{dancer@maths.ox.ac.uk}

\author{McKenzie Y. Wang}
\address{Department of Mathematics and Statistics, McMaster
     University, Hamilton, Ontario, L8S 4K1, CANADA}
\email{wang@mcmaster.ca}
\thanks{M.Wang is partially supported by NSERC Grant No. OPG0009421}

\date{revised \today}

\begin{abstract}
We show how to view the equations for a cohomogeneity one Ricci soliton
as a Hamiltonian system with a constraint. We investigate conserved quantities
and superpotentials, and use this to find some explicit formulae for
Ricci solitons not of K\"ahler type in five dimensions.
\end{abstract}

\maketitle


\bigskip
\setcounter{section}{0}

\section{\bf Introduction}

Einstein metrics are the critical points of the Einstein-Hilbert action
restricted to the set of Riemannian metrics with fixed volume. In the
context of metrics of cohomogeneity one, this variational characterization
and the invariance under diffeomorphisms give rise to an interesting structure
for the Einstein equations--they can be written as a Hamiltonian system with a
constraint (the vanishing of the Hamiltonian). This structure has proved useful
in identifying conserved quantities for certain cases of the Einstein system,
and for finding superpotentials which define first order subsystems of the Einstein
equations (see \cite{DW1}, \cite{DW2}, \cite{DW6} for example). Frequently,
these quantities in turn lead to explicit solutions of the Einstein equations,
and the associated first order subsystems single out solutions with special holonomy.

In an analogous manner, Perelman's $\mathcal{F}$ and $\mathcal{W}$ functionals (\cite{Pe})
possess diffeomorphism and scale invariance properties, and consideration of
their first variations which preserve the dilaton measure leads to the gradient
Ricci soliton equations. In this paper we will investigate the cohomogeneity one
case using the framework of \cite{DW5} and put the gradient Ricci soliton equations
in Hamiltonian form with a constraint. We then focus on the case of steady solitons
and consider the situation in which the principal orbit has multiplicity free isotropy
representation.

Let us denote the Hamiltonian (to be constructed in \S 3) by $\mathscr H$. Recall that
a {\em generalised first integral} is a function on momentum phase space whose Poisson
bracket with the Hamiltonian lies in the ideal generated by $\mathscr H$. As well, a
{\em superpotential} is a $C^2$ function $f$ on configuration space that gives a time-independent
solution of the Hamilton-Jacobi equation, i.e., ${\mathscr H}(q, df_q) = 0$.
In the situation of the Bryant solitons, we find both generalised first integrals and
superpotentials when the dimension is either $2$ or $5$. This leads to
new explicit formulae for the Bryant soliton in dimension $5$. We also find superpotentials
for the gradient Ricci soliton equations on double warped products $\R^3 \times \Sigma$
where $\Sigma = S^2$ or $\R\PP^2$ and for certain complex line bundles over a product of
Fano K\"ahler-Einstein manifolds. The associated first order systems in the latter case
correspond to the K\"ahler condition, and explicit complete steady K\"ahler Ricci solitons
were obtained in \cite{DW5} (Theorem 4.20). Here we use the superpotentials in the former
case to obtain explicit complete steady gradient Ricci soliton structures on
$\R^3 \times \Sigma$. The existence of these non-K\"ahlerian solitons was obtained
previously using dynamical systems methods (cf \cite{Iv} and \cite{DW3}), but
{\em explicit} complete solutions had not, as far as we are aware, been found previously.

Finally, we mention that in \S 5 we prove a non-existence theorem (see Proposition \ref{nonnull})
for superpotentials of exponential type assuming a certain natural condition. We refer
the reader to \S 5 for the description of this condition. It is noteworthy, however, that
superpotentials satisfying this condition do exist in the Ricci-flat case and most of the
time when they occur, the associated first order systems are equivalent to the condition
that the metrics have special holonomy. So this non-existence result may be interpreted as
an indication of the greater rigidity of the soliton equation.

\section{\bf Basic Facts}

Let $(M, \bar{g})$ denote a connected Riemannian manifold of dimension
$n+1$ on which a compact Lie group $G$ acts via isometries with cohomogeneity one, i.e.,
with one-dimensional orbit space. We will assume that the orbit space is an interval
$I \subset \R$ and that the principal orbit type is given by $P:=G/K$, where $K$ is a closed
subgroup of $G$. We further assume that there is at least one special orbit, which, without loss of
generality, is of the form $G/H$ where $K \subset H$ and $H$ is closed in $G$. The
cohomogeneity one condition then implies that $H/K$ is diffeomorphic to a sphere $S^k$. By
choosing a constant speed geodesic that intersects one (hence all) principal orbits orthogonally,
we obtain a diffeomorphism of the open dense submanifold $M_0 \subset M$ consisting of all
the principal orbits with ${\rm int}\, I \times P$. Then the metric $\bar{g}$ takes the form
\begin{equation} \label{met-form}
\bar{g} = dt^2 + g_t
\end{equation}
where $t$ is the geodesic parameter and $g_t$ is a one-parameter family of $G$-invariant
metrics on $G/K$.

Let $L_t$ denote the shape operator of the hypersurface $\{t\}  \times P$, regarded
as a $g_t$-symmetric endomorphism of $T(G/K)$. Also, let $r_t$ denote the Ricci
endomorphism defined by $\ric(g_t)(X, Y) = g_t(r_t (X), Y))$, where $X, Y$ are tangent to
$G/K$. Then $\dot{g}_t = 2 g_t \circ L_t$, where each $g_t$ is regarded as an endomorphism
of $T(G/K)$ via the choice of a fixed $G$-invariant metric $Q$ on $G/K$. The Levi-Civita
connections of $\bar{g}$ (resp. $g_t$) will be denoted by $\overline \nabla$ (resp. $\nabla$).
The relative volume $v$ is defined by $d\mu_{g_t} = v(t)\, d\mu_Q$.

The static equation for a gradient Ricci soliton (GRS) $(M, \bar{g}, u)$ is
\begin{equation} \label{GRS}
    {\rm Ric}(\bar{g}) + {\rm Hess}_{\bar{g}} \,u + \frac{\epsilon}{2} \,\bar{g} = 0.
\end{equation}
$\bar{g}$ is called the {\em soliton metric} and $u: M \rightarrow \R$ is a smooth
function called the {\em soliton potential}. We shall say that a gradient Ricci soliton
is {\em trivial} if the soliton metric is Einstein. The example of the Gaussian soliton
shows that the potential could nevertheless be non-trivial.

In the cohomogeneity one situation,  Eq. (\ref{GRS})  becomes,
on $M_0 \approx {\rm int} \,I \times P$, the system
\begin{eqnarray}
r_t -\dot{L} - ({\rm tr}\, L -\dot{u})\, L + \frac{\epsilon}{2} \,\I &=& 0
\label{soleqS}, \\
 - {\rm tr} (L^2) -{\rm tr}\, (\dot{L}) + \ddot{u} + \frac{\epsilon}{2}&=& 0,
\label{soleqN} \\
d({\tr L}) + \delta^{\nabla} L &=& 0,
\label{soleq-mixed}
\end{eqnarray}
where the soliton potential $u$ is regarded both as a function on $M$ and as a function of $t$,
and $\delta^{\nabla}$ is the codifferential for $T^*(G/K)$-valued $1$-forms.

Note that Eq.(\ref{soleqS}) above represents the components of the gradient Ricci soliton equation
tangent to $G/K$, Eq.(\ref{soleqN}) is the equation in the $\partial/\partial t$ direction, and
Eq.(\ref{soleq-mixed}) represents the mixed directions. In fact, by $G$-invariance, $\tr \,L$ is
always constant in the $G/K$ directions, but we have included it in (\ref{soleq-mixed}) because
the above system also holds under suitable assumptions when $M$ is constructed out of an equidistant
family of hypersurfaces with possibly {\em no} symmetries (cf \cite{DW4}, Remark 2.18).
This possibility, incidentally, is a frequently misunderstood aspect of the work of the last two authors.

For the general GRS equation, there is a fundamental {\it conservation law} which was observed by R. Hamilton
\cite{Ham} and T. Ivey \cite{Iv}. In the cohomogeneity one setting, this conservation law becomes
\begin{equation} \label{cons1}
\ddot{u} +  ( - \dot{u} +{\rm tr}\, L )\,\dot{u} -\epsilon u = C
\end{equation}
where $C$ is a fixed constant. Using (\ref{soleqN}) and the trace of (\ref{soleqS}), we can
rewrite (\ref{cons1}) in the form
\begin{equation} \label{ham}
S + \tr(L^2)- (-\dot{u}+ \tr \,L)^2 + (n-1)\frac{\epsilon}{2} = C + \epsilon u,
\end{equation}
where $S$ is the scalar curvature of the principal orbits and $C$ is the same constant.
Recall that in \cite{DHW} we introduced the quantities
\begin{equation} \label{xi-def}
  \xi := -\dot{u} + \tr \, L, \  \ \  \ \ \  {\mathcal E} = C + \epsilon u,
\end{equation}
and rewrote the conservation law in the form
\begin{equation} \label{eqnE}
\ddot{\mathcal E} + \xi \dot{\mathcal E} - \epsilon {\mathcal E} = 0.
\end{equation}
Note that $\xi$ is a natural quantity to consider, as it is just the mean curvature of the
dilaton volume element $e^{-u} d\mu_{\bar{g}}$. Furthermore, for functions of the variable $t$,
the operator $\ddot{f} + \xi \dot{f}$ is just the $u$-Laplacian in the theory of metric
measure spaces.

Finally we recall that the scalar curvature of the metric $\bar{g}$ is given by
\begin{equation} \label{ambientR}
\bar{R} = -2\ \tr(\dot{L}) - \tr(L^2) - (\tr L)^2 + S.
\end{equation}
Using the trace of (\ref{soleqS}) followed by (\ref{ham}) we obtain
\begin{eqnarray}
 \bar{R} &=& -S + (\tr L)^2 - \tr(L^2) -2 \dot{u} \ \tr L - \epsilon n \label{ambientR1} \\
   &=& -C -\epsilon u - {\dot{u}}^2 - \frac{\epsilon}{2} (n+1),    \label{ambientR2}
\end{eqnarray}
which is just the cohomogeneity one case of Hamilton's identity (cf \cite{Ham} p. 84).

\section{\bf A Hamiltonian Formulation of the Cohomogeneity One GRS Equations.}

In this section we will construct a Hamiltonian on an appropriate symplectic manifold
such that integral curves of the associated Hamiltonian vector field lying on the zero
energy hypersurface correspond to solutions of the cohomogeneity one gradient Ricci soliton
equations modulo smoothness considerations. The Einstein case was discussed in \cite{DW1}
and the present case is essentially analogous.

Given a principal orbit $G/K$, we first fix an ${\rm Ad}_K$-invariant decomposition
$$\g = \kf \oplus \p$$
of the Lie algebra $\g$ of $G$, so that $\p \approx T_{[K]}(G/K)$. Recall that in \S 2 we have fixed
a background invariant metric $Q$ on $\p$.  Now let $\mathscr C$ be the {\it configuration space}
$S_{+}^2(\p)^K \times \R$ where $S_{+}^2(\p)^K$ denotes the space of all ${\rm Ad}(K)$-invariant,
positive-definite, symmetric endomorphisms of $\p$ with respect to $Q$. Via the relation
$g_t(X, Y) = Q(q_t(X), Y)$, a path $(q_t, u(t)) \in {\mathscr C}$ corresponds to a one-parameter
family of $G$-invariant metrics on $G/K$ together with a soliton potential function.

The {\it velocity phase space} is $T{\mathscr C} = (S_{+}^2(\p)^K \times \R) \times (S^2(\p)^K \times \R)$
and we will denote a typical element in it by $(q, u, \dot{q}, \dot{u})$. In order to write down
a suitable Lagrangian function, we introduce the following non-degenerate symmetric bilinear form
on ${\rm End}(\p) \times \R$:
\begin{equation} \label{sym-form}
  \langle (h_1, \eta_1), (h_2, \eta_2) \rangle := \frac{1}{2}\left(\tr(h_1)\,\tr(h_2) -\tr(h_1 h_2)\right)
       + 2 \eta_1 \eta_2 - \left(\tr(h_1)\,\eta_2 + \tr(h_2)\, \eta_1 \right).
\end{equation}
The induced symmetric bilinear form on ${\rm End}(\p)^* \times \R^*$ will be denoted by
$\langle \cdot, \cdot \rangle^*$. These forms are, {\em up to a minus sign}, extensions of the
symmetric bilinear forms on ${\rm End}(\p)$ and ${\rm End}(\p)^*$ introduced in \cite{DW1} (see
Eq. (1.15) there). As is easily checked, the extended forms also have Lorentz signature
$(-, \cdots, -, +)$.

Recall also that ${\rm Gl}(\p)$ acts on the left of ${\rm End}(\p)$ (by composition), and we
can extend this action  to ${\rm End}(\p) \times \R$ by making ${\rm Gl}(\p)$ act trivially on
$\R$. In particular,  for $(p, \phi) \in S^2(\p^*)^K \times \R^*$, $q \in S_{+}^2(\p)^K$, and
$(h, \eta) \in S^2(\p)^K \times \R$, we have
\begin{equation} \label{action}
(q^{-1} \cdot (p, \phi))(h, \eta) = p(q\cdot h) + \phi(\eta).
\end{equation}

We now introduce the Lagrangian
\begin{equation} \label{lagr}
{\mathscr L}(q, u, \dot{q}, \dot{u}) = e^{-u}v(q)\left( \tau \left(  \frac{1}{2}\langle q^{-1}\dot{q},
      q^{-1}\dot{q} \rangle + \dot{u}^2 - \dot{u}\, \tr(q^{-1}\dot{q}) + S(q) \right) + \lambda(u-n-1) + E \right)
\end{equation}
where $v(q)$ is the relative volume and $S(q)$ the scalar curvature associated to the metric $q$ on $G/K$.
The parameter $E$ is a Lagrange multiplier associated to the constant energy condition that will be
presently introduced, and the parameters $\tau$ and $\lambda$ will eventually be set equal  to $1$
and $-\epsilon$ respectively.

\begin{rmk} \label{W-func}
The above Lagrangian can be derived from Perelman's $\mathscr W$-functional (\cite{Pe})
$$ {\mathscr W}(\bar{g}, u, \tau) = \frac{1}{(4\pi \tau)^{N/2}} \int_M \, \left( \tau \left(\bar{R} + |\nabla u|^2 \right)
      - \epsilon( u - N) \right) e^{-u} d\mu_{\bar{g}}, $$
in which $\bar{R}$ is the scalar curvature of $\bar{g}$ and $N=\dim M$. Strictly speaking,
Perelman considered only the $\epsilon = -1$ case for compact $M$.  The above modification was
introduced for example in \cite{Cetc}, see p. 229.

In the situation where $M$ is a cohomogeneity one manifold, we write its dimension $N$ as $n+1$
and substitute Eqn. (\ref{ambientR}) into the above integral. Recall that $\tr \,L$ is the logarithmic
derivative of the relative volume $v$. Then, integrating by parts formally to
get rid of second derivative terms, one obtains (\ref{lagr}).  The usual constraint
$$ \frac{1}{(4 \pi \tau)^{N/2}} \int_M \, e^{-u} d\mu_{\bar{g}} = {\rm const}$$
is accounted for by the introduction of the multiplier $E$ above.  In order to treat the
cases of steady and expanding solitons simultaneously with the shrinking case, we have
further introduced the parameter $\lambda= -\epsilon$ and have suppressed the multiplicative factor
${\rm vol}(Q)/(4 \pi \tau)^{n+1}$ in $\mathscr W$. (Note that $d\mu_{\bar{g}} = v(q) \vol(Q) dt$.)
An important advantage of using the alternative Lagrangian (\ref{lagr}) as a starting point
in the Hamiltonian approach is that we can take $M$ to be {\em non-compact or even incomplete}.
\end{rmk}

The {\em momentum phase space} is the cotangent bundle
$T^* {\mathscr C} = (S_{+}^2(\p)^K \times \R) \times (S^2(\p^*)^K \times \R^*)$ equipped with
the canonical symplectic structure. A typical element of $T^* {\mathscr C}$ will be
denoted by $(q, u, p, \phi)$. The Legendre transformation is defined by the equations
\begin{eqnarray}
   p(h) &=& {\mathscr L}_{\dot{q}}(h) = e^{-u} v(q) \,\tau \left(\langle q^{-1}\dot{q}, q^{-1}h \rangle
         - \dot{u}\, \tr(q^{-1}h) \right),  \label{legen1}  \\
   \phi(\eta) &=& {\mathscr L}_{\dot{u}}(\eta) = e^{-u} v(q) \,\tau
           \left(2 \dot{u}\eta - \tr(q^{-1}\dot{q}) \eta \right).    \label{legen2}
\end{eqnarray}
The associated Hamiltonian is then given by
$${\mathscr H} = p(\dot{q}) + \phi(\dot{u}) - {\mathscr L}.$$
More explicitly, $\mathscr H$ can be written as
\begin{equation} \label{momentum-ham}
{\mathscr H} = \frac{e^{u}}{v(q)} \frac{\tau}{2} \langle q^{-1} \cdot (p, \phi), q^{-1} \cdot (p, \phi) \rangle^*
           + \frac{v(q)}{e^{u}} \left( -E + \lambda(n+1 -u) - \tau S(q) \right),
\end{equation}
where we have used the fact that the covector $q^{-1}\cdot (p, \phi)$ is dual to
$e^{-u}v(q)(q^{-1}\dot{q}, \dot{u})$ with respect to the symmetric bilinear form (\ref{sym-form}).
The above Hamiltonian should be compared to (1.9) in \cite{DW1}.
Equivalently, via the inverse Legendre transformation,
 we have
\begin{equation} \label{velocity-ham}
{\mathscr H} = v(q)e^{-u} \left( \tau \left(2 \langle L, L \rangle + \dot{u}^2 -2\dot{u}\, \tr L  \right)
      - E + \lambda(n+1 -u) - \tau S(q) \right).
\end{equation}
Since (\ref{sym-form}) specializes to $2 \langle L, L \rangle = (\tr L)^2 - \tr (L^2)$,
the {\em zero energy condition} can be rewritten as
\begin{equation} \label{0-energy}
\tr (L^2) - (-\dot{u} + \tr L)^2 + S + \frac{\lambda}{\tau} \,u = \frac{1}{\tau}(-E + \lambda (n+1)),
\end{equation}
which is just the conservation law (\ref{ham}) if we set $\tau =1$ and $\lambda = -\epsilon$.
The constant $C$ in (\ref{ham}) is then given by $-(E+ \frac{\epsilon}{2} (n+3))$.

\begin{rmk} \label{entropy}
In the case of a gradient Ricci soliton of cohomogeneity one, if we use Eq. (\ref{ambientR2})
together with the above relation between $E$ and $C$ in the $\mathscr W$-functional, we obtain
$$ {\mathcal W}(\bar{g}, u, 1) = \frac{1}{(4\pi)^{(n+1)/2}} \int_M \, \left(E + \epsilon(n+2)
         - 2\epsilon u \right) e^{-u} d\mu_{\bar{g}}. $$
Let us choose the normalization $ (4\pi)^{-(n+1)/2} \int_M \, e^{-u}  d\mu_{\bar{g}} = 1$.
For the steady case (i.e., $\epsilon = 0$) one can then interpret $E$ as Perelman's energy
${\mathcal F}(\bar{g}, u)$. For the shrinking case ($\epsilon < 0$), we obtain
$$ {\mathcal W}(\bar{g}, u, 1) = \left(E + \frac{\epsilon}{2} (n+2)\right) -
     \frac{2\epsilon}{(4\pi)^{(n+1)/2}} \int_M \, u e^{-u} d\mu_{\bar{g}} $$
where the last integral is the classical entropy.
\end{rmk}

As in \cite{DW1}, the geometric significance of the zero energy condition is given by

\begin{prop} \label{zeroE}
Assume that there is a singular orbit $G/H$ with dimension strictly smaller than that
of the principal orbits. An integral curve of the Hamiltonian vector field that corresponds $($under
the Legendre transformation $)$ to a $C^2$ Riemannian metric $\bar{g} = dt^2 + g_t$ and a
potential function $u(t)$ defined in an open neighbourhood of the singular orbit must actually lie
on the variety $\{{\mathscr H} = 0\}$.
\end{prop}

\begin{proof} We may assume that the singular orbit is placed at $t=0$. The Hamiltonian is constant
along any integral curve. To evaluate the value of the constant, note that the smoothness
conditions imply that $v(0)=0, \dot{u}(0)=0,$ and $u(0)$ is finite. By the proof of Lemma 1.10
in \cite{DW1}, as $t$ tends to $0$, all terms on the right of (\ref{velocity-ham}) tend to $0$
except possibly the term involving $ v \dot{u}\, \tr L$. But $\dot{u}\ \tr L$ tends to a finite
constant, so in fact the remaining term also tends to $0$.
\end{proof}

Regarding the zero set ${\mathscr Z}_{\mathscr H}:=\{ {\mathscr H} = 0 \}$, we have

\begin{prop} \label{zeroset} When $\lambda \neq 0$ the variety ${\mathscr Z}_{\mathscr H}$
is a smooth hypersurface in momentum phase space. It is also smooth when $\lambda = 0$
and the principal orbit $G/K$ is not a torus; otherwise the possible singular points are of the form
$(q, u, p, \phi) = (q, u, 0, 0)$ where $q$ corresponds to a $G$-invariant flat metric on $G/K$.
\end{prop}

\begin{proof} We need to examine the differential $d{\mathscr H}$ at points $(q, u, p, \phi)$
in ${\mathscr Z}_{\mathscr H}$. The partial derivative
$$ {\mathscr H}_{(p, \phi)}(\alpha, \beta) = \frac{e^u}{v}\, \tau \,\langle q^{-1}\cdot(p, \phi),
         q^{-1}\cdot(\alpha, \beta) \rangle^*.$$
Since $\langle \, \, , \, \rangle$ is non-degenerate, the above partial derivative vanishes
(for all $(\alpha, \beta)$) iff $(p, \phi) = (0, 0)$.  The vanishing of $\mathscr H$ now implies
that $E = -(\tau S + \lambda (u -n-1))$. Using these two facts in the partial derivative of
${\mathscr H}$ with respect to $u$ we obtain
$$ {\mathscr H}_u (\eta) = -\lambda \,\eta \, v e^{-u},$$
which vanishes (for all $\eta$) only if $\lambda = 0$. This gives the first statement of the Proposition.

If $\lambda = 0$, using $(p, \phi) = (0, 0)$, we obtain
$$ {\mathscr H}_q (h) = e^{-u} dv_q(h) (-E + \lambda(n+1 -u) - \tau S) - e^{-u} v (\tau (dS)_q(h)). $$
Since the first term vanishes by the zero energy condition, the vanishing of ${\mathscr H}_q$
reduces to the vanishing of $(dS)_q$. As in the proof of Proposition 1.15 in \cite{DW1},
we conclude that $q$ is a Ricci-flat $G$-invariant metric on the principal orbit. It is well-known
that the principal orbit must then be a torus.
\end{proof}

Since the solution curves of the Euler-Lagrange equation for $\mathscr L$ correspond
to the integral curves of the canonical equations for $\mathscr H$, we proceed to
determine explicitly the components of the Euler-Lagrange equation and show that they
yield (\ref{soleqS}) and (\ref{soleqN}) if one further assumes the zero energy condition.

\medskip

We begin with
$$ {\mathscr L}_{(\dot{q}, \dot{u})} (h, \eta) = \tau v e^{-u} \left( 2\dot{u} \eta
       - 2 (\tr L)\,\eta + 2 \langle L, q^{-1}h \rangle - \dot{u}\, \tr(g^{-1}h) \right). $$
It follows that
\begin{eqnarray*}
\frac{d}{dt} {\mathscr L}_{(\dot{q}, \dot{u})} (h, \eta) &=& \tau v e^{-u} \left[
          (-\dot{u} + \tr L)(2 \eta \left(\dot{u} - \tr L) + 2\langle L, q^{-1}h \rangle - \dot{u}\, \tr(q^{-1}h)\right) \right.\\
      &  &  \left. + 2\eta(\ddot{u} - \tr(\dot{L})) + 2 \langle \dot{L}, q^{-1}h \rangle
             -2 \langle L, q^{-1}\dot{q}q^{-1} h \rangle - \ddot{u}\, \tr(q^{-1}h) +
             \dot{u}\, \tr(q^{-1}\dot{q}q^{-1} h)  \right].
\end{eqnarray*}
We also have
\begin{eqnarray*}
{\mathscr L}_{(q, u)} (h, \eta) &=& v e^{-u} \left( \frac{1}{n-1} \langle I, q^{-1}h \rangle - \eta \right) \left[
       \tau \left(2 \langle L, L \rangle + \dot{u}^2  -2 \dot{u}\, \tr L + S \right) + \lambda(u-n-1) + E \right] \\
     &  & + v e^{-u} \left[ \lambda \eta + \tau \left( 2 \langle q^{-1}h, r_q \rangle - \tr(q^{-1}h) S
           + 2 \dot{u} \, \tr(q^{-1}h L) - 4 \langle L, q^{-1}h L \rangle  \right)\right],
\end{eqnarray*}
where we have used Lemma 1.12 in \cite{DW1} and $r_q$ denotes the Ricci endomorphism of the metric $q$.

Now the Euler-Lagrange equation
$$\frac{d}{dt} {\mathscr L}_{(\dot{q}, \dot{u})} (h, \eta) = {\mathscr L}_{(q, u)} (h, \eta)$$
must hold for all $h \in S^2(\p)^{K}$ and all $\eta \in \R$. Setting $h=0$ and simplifying, we obtain
the equation
\begin{equation} \label{presoleqN}
2 \ddot{u} -2 \tr(\dot{L}) - (-\dot{u} + \tr L)^2 - \tr(L^2) + S + \frac{\lambda}{\tau} u
         = \frac{-E + \lambda(n+2)}{\tau}.
\end{equation}
If we then apply the zero energy condition (\ref{0-energy}), we obtain
\begin{equation} \label{altsoleqN}
     \ddot{u} - \tr(\dot{L}) - \tr(L^2) = \frac{\lambda}{2\tau},
\end{equation}
which becomes Eq. (\ref{soleqN}) if we set $\tau=1$ and $\lambda = -\epsilon$.

If instead we set $\eta = 0$ in the Euler-Lagrange equation, then after some amount of simplification
we obtain the equation
\begin{equation} \label{presoleqS}
 \dot{L} + (\tr L)L - \dot{u}\, L - r = \left( 2 \langle L, L \rangle + 2 \ddot{u} - \dot{u}^2
      -S + \frac{\lambda(u -n-1) + E}{\tau}  \right) \frac{I}{2(n-1)}.
\end{equation}
Taking the trace of (\ref{presoleqS}) yields
\begin{equation*}
\tr(\dot{L}) + (\tr L)^2 - \dot{u}\, (\tr L) - S = \frac{n}{2(n-1)} \left( (\tr L)^2 - \tr(L^2)
         + 2 \ddot{u} - \dot{u}^2 -S + \frac{\lambda(u-n-1) + E}{\tau} \right).
\end{equation*}
If we substitute (\ref{altsoleqN}) and the zero energy condition (\ref{0-energy}) into this
equation, then after some simplication we deduce
$$ 2 \ddot{u} - \dot{u}^2 - S + (\tr L)^2 - \tr(L^2) + \frac{\lambda}{\tau} u = \frac{-E + 2 \lambda}{\tau}. $$
Substituting this equation into (\ref{presoleqS}) gives
\begin{equation} \label{altsoleqS}
 \dot{L} + (\tr L)L - \dot{u} L -r = -\left(\frac{\lambda}{2\tau}\right) I,
\end{equation}
which becomes Eq. (\ref{soleqS}) if we set $\tau=1$ and $\lambda = -\epsilon.$
We have therefore deduced

\begin{thm} \label{hamform}
Given a principal orbit $G/K$ where $G$ is a compact Lie group and $K$ a closed subgroup,
consider on the symplectic manifold $(S_{+}^2(\p)^K \times \R) \times (S^2(\p^*)^K \times \R^*)$
the Hamiltonian ${\mathscr H}$ given by $($\ref{momentum-ham}$)$ with $\tau=1$ and
$\lambda=-\epsilon$. The integral curves of  ${\mathscr H}$ lying in in the variety
$\{{\mathscr H} = 0\}$ correspond $($under the inverse Legendre transformation$)$ to
solutions of the non-mixed parts of the cohomogeneity one gradient Ricci soliton equations
$($\ref{soleqS}$)$ and $($\ref{soleqN}$)$.   \hspace{5cm}$\Box$
\end{thm}

Recall that by Proposition 3.19 in \cite{DW5}, as long as there is a singular orbit
of dimension strictly less than that of the principal orbit and we can further establish
$C^3$ regularity of the metric $\bar{g} = dt^2 + g_t$ and potential $u$, then the mixed
parts of the cohomogeneity one GRS equation automatically hold. In this sense the
cohomogeneity one GRS equation can be viewed as a constrained Hamiltonian system.

\medskip

In the remainder of this section we will derive a more explicit form of the Hamiltonian
$\mathscr H$ in the special case where the isotropy representation of $G/K$ splits into
pairwise inequivalent irreducible $\R$-subrepresentations.

To this end let us write
\begin{equation} \label{p-decomp}
 \p = \p_1 \oplus \cdots \oplus \p_r
\end{equation}
for the decomposition of $\p$ into ${\rm Ad}(K)$-irreducible $Q$-orthogonal summands
and let $d_i = \dim_{\R} \p_i$, so that $n = \sum_i d_i$. We will abuse notation and
denote the metric endomorphism $q \in S^2_{+}(\p)^{K}$ by the diagonal operator
${\rm diag}(e^{q_1} I_{d_1}, \cdots, e^{q_r} I_{d_r})$ where $I_{d_i}$ is the
identity operator in ${\rm End}(\p_i)$. In other words, via these new coordinates,
we have a diffeomorphism $S^2_{+}(\p)^{K} \approx \R^r$ which in turn induces a canonical
transformation of $T^*{\mathscr C}$, leaving the remaining variables $u, \phi$ unchanged.
By abuse of notation we shall let $p_i$ denote the new conjugate momenta. It will be
useful to let $q, p, d$ denote the vectors in $\R^r$ whose coordinates are respectively
$q_i, p_i, d_i$.

Then, as in the Einstein case, we have $v=\exp(\frac{1}{2} d \cdot q)$,
$$ \tr L = \frac{1}{2} \,d \cdot \dot{q}, \,\,\,\,\, \tr(L^2) = \frac{1}{4} \sum_i \, d_i \,\dot{q}_i^2,$$
and
$$ S = \sum_{w \in {\mathcal W}} \,\, A_w \, e^{w \cdot q} $$
for a finite subset ${\mathcal W} \subset \R^r$ of weight vectors and nonzero real constants $A_w$
which depend only on $G/K$. (See \S 1 in \cite{DW6} for further information about $\mathcal W$ and $A_w$.)

With the above change of coordinates, the Lagrangian (\ref{lagr}) becomes
$${\mathscr L} = e^{-u+\frac{1}{2} d\cdot q} \left( \frac{1}{4}(d \cdot {\dot q})^2 - \frac{1}{4} \sum_i \, d_i {\dot q}_i^2
         + \dot{u}^2 - \dot{u}\, (d \cdot \dot{q}) + \sum_{w\in {\mathcal W}} \, A_w  e^{w \cdot q} + E \right).$$

The Legendre transformation is now given by
$$ p_j = {\mathscr L}_{\dot{q}_j} =  e^{-u+\frac{1}{2} d\cdot q} \left( \frac{1}{2}\,(d \cdot {\dot q})\,d_j
           - \frac{1}{2}\, d_j \,\dot{q}_j  - \dot{u}\, d_j \right)$$
and
$$ \phi = {\mathscr L}_{\dot{u}} = e^{-u+\frac{1}{2} d\cdot q} (2 \dot{u} - d\cdot {\dot q}).$$

Using the above equations, we easily deduce that
$$ \frac{1}{2}\, d \cdot \dot{q} = -\left( \sum_j \, p_j + \frac{n}{2} \phi \right) e^{u-\frac{1}{2} d\cdot q}, $$
$$ \dot{u} = -\left( \sum_j \, p_j  + \frac{n-1}{2} \phi \right)e^{u-\frac{1}{2} d\cdot q},$$
$$ \frac{1}{4} \sum_i \, d_i \dot{q}_i^2 = \left(\sum_i \, \frac{p_i^2}{d_i}
      + \phi \sum_i \, p_i + \frac{n}{4} \phi^2 \right) e^{2u-d\cdot q}  .$$

Substituting the above relations in ${\mathscr H} = p(\dot{q}) + \phi(\dot{u}) - {\mathscr L}$
we obtain an explicit formula for $\mathscr H$ in the multiplicity free case:

\begin{prop} \label{ham-0mult}
Let $G/K$ be a principal orbit whose isotropy representation splits into pairwise inequivalent
irreducible summands. In terms of the exponential coordinates introduced above, the Hamiltonian
$($\ref{momentum-ham}$)$  takes the form
\[ {\mathscr H} =  - \frac{e^{u-\frac{1}{2} d\cdot q}}{\tau} \left(\sum_i \, \frac{p_i^2}{d_i} +
           \phi \sum_i \, p_i  + \frac{(n-1)}{4} \phi^2 \right)
       +  e^{-u+\frac{1}{2} d\cdot q} \left(- E + \lambda(n+1 -u)
       -  \tau \sum_{w \in {\mathcal W}} \, A_w e^{w \cdot q} \right).
\]
\end{prop}

\begin{rmk} The above proposition is the analogue of Proposition 2.5 in \cite{DW1}. The quadratic
form
$$ J(p, \phi) := - \sum_i \, \frac{p_i^2}{d_i} - \phi \sum_i \, p_i  - \frac{(n-1)}{4}\, \phi^2 $$
assumes a different form from its analogue in \cite{DW1}, but it is still of Lorentz signature
$(1, r)$. It is negative definite on the hyperplane $\phi=0$, but is positive on the
vector $(d, -2)$ (in fact $J(d, -2) =1$.)

The associated bilinear form is
\begin{equation} \label{bilnform}
-\left( \sum_{i=1}^{r} \frac{p_i p_i^\prime}{d_i} + \frac{\phi}{2}
\sum_{i=1}^{r} \,p_i^\prime + \frac{\phi^\prime}{2} \sum_{i=1}^{r} \,p_i
+ \frac{n-1}{4} \phi \,\phi^\prime \right).
\end{equation}
\end{rmk}

It is sometimes convenient to rewrite the Hamiltonian using the extended vectors
\[
{\bf d} := (d, -2) = (d_1, \ldots, d_r, -2),
\]
\[
{\bf q } := (q,u) = (q_1, \ldots, q_r, u),
\]
\[
{\bf p} := (p, \phi) = (p_1, \ldots, p_r, \phi).
\]
Similarly the weight vectors $w$ can be extended to ${\bf w} = (w,0)$, and we will
continue to use $\mathcal W$ to denote the set of extended vectors in $\R^{r+1}$.
The Hamiltonian now becomes
\[
 {\mathscr H} =   \frac{e^{-\frac{1}{2} {\bf d} \cdot {\bf q}}}{\tau} J({\bf p})
       +  e^{\frac{1}{2} {\bf d} \cdot {\bf q}} \left( - E + \lambda(n+1 -u)
       -  \tau \sum_{w \in {\mathcal W}} \, A_w e^{{\bf w} \cdot {\bf q}} \right)
\]
where, as above,
\[
J({\bf p})=
-\left(\sum_i \, \frac{p_i^2}{d_i} +
           \phi \sum_i \, p_i  + \frac{(n-1)}{4} \phi^2 \right).
\]
In the steady case $(\lambda=0$), if we also set $\tau=1$, we get
\begin{equation}
{\mathscr H} =   e^{-\frac{1}{2} {\bf d} \cdot {\bf q}} J({\bf p})
       -  e^{\frac{1}{2} {\bf d} \cdot {\bf q}}  \left( E
       +  \sum_{w \in {\mathcal W}} \, A_w e^{{\bf w} \cdot {\bf q}} \right). \label{ham-steady}
\end{equation}
It is also often useful to enlarge the set $\mathcal W$
to $\tilde{\mathcal W} = {\mathcal W} \cup \{0\}$, so that the
final bracket in Eq. (\ref{ham-steady}) is viewed as a sum of exponentials over
$\tilde{\mathcal W}$, with $E$ playing the role of $A_0$.

\section{\bf Conserved quantities}

As in the Einstein case we can look for quantities
$F$ which are {\em generalised first integrals}, in the sense that
\begin{equation}\label{conseq}
\{ F, {\mathscr H} \} = \Phi \mathscr H
\end{equation}
so that for the soliton equations, that is, the Hamiltonian
flow in the variety ${\mathscr H}=0$, the quantity $F$ is conserved.
Functions in the ideal generated by $\mathscr H$ will be referred to as
trivial generalised first integrals.

We look for solutions to (\ref{conseq})
of the form
\[
F = \sum_{\bf b} F_{\bf b}\, e^{{\bf b} \cdot {\bf q}} \;\; : \;\;
\Phi = \sum_{\bf b} \Phi_{\bf b}\, e^{{\bf b} \cdot {\bf q}}
\]
where $F_{\bf b}$ and $\Phi_{\bf b}$ are polynomials in $\bf p$. In view
of this we shall in this section
denote by $\nabla_{\bf p}$ the gradient operator in
the momentum variables whenever the possibility of confusion may occur.
In cases where the operator is applied to functions such as $J, F$ and
$\Phi$ which depend solely on ${\bf p}$ we shall suppress this subscript.

Substituting these into (\ref{conseq}), and setting $\psi := \Phi -
\frac{1}{2}\, {\bf d} \cdot \nabla_{\bf p} F$  we obtain, upon using
(\ref{ham-steady}),  the recursion relation
\begin{equation} \label{recursion}
({\bf b} \cdot \nabla J) F_{\bf b} - \psi_{\bf b} J =  -E(\psi_{{\bf b} - {\bf d}}
+ {\bf d} \cdot \nabla F_{{\bf b} -{\bf d}})
-\sum_{w} A_w (\psi_{{\bf b} - {\bf d} -{\bf w}} +({\bf d}+{\bf w})\cdot \nabla
F_{{\bf b} -{\bf d} -{\bf w}}).
\end{equation}

We adopt a similar strategy as in \cite{DW1} to look for nontrivial
generalised first integrals, starting with a seed level ${\bf c}$
where there is a factorisation
\[
J = ({\bf c} \cdot \nabla J)  \theta
\]
so we may obtain a nontrivial solution to the recursion
(i.e. one where we do not have $F_{\bf c} = JG$ and
$\psi_{\bf c} = ({\bf c}. \nabla J)G$) at this level by setting
\[
F_{\bf c} = \theta \psi_{\bf c}.
\]

Let us consider the Bryant soliton where the hypersurface is just the
sphere $S^n$, viewed as the isotropy irreducible space $SO(n+1)/SO(n)$.
(In the literature, using this particularly simple form of the sphere in
the cohomogeneity one ansatz is referred to as the {\em rotationally symmetric} case.)
We write the soliton metric as
$$ \bar{g} = dt^2 + h(t)^2 g_1$$
where $g_1$ is the constant curvature one metric on $S^n$. The scalar curvature
function for the sphere is $n(n-1) e^{-q_1}$, where we set $h(t)^2 = e^{q_1}$.
Writing $q$ for $q_1$, we have:
\[
{\bf d} = (n, -2), \;\; {\bf q} = (q, u), \;\; {\bf p} = (p, \phi),
\]
\[
\tilde{\mathcal W} = \{ (0,0), (-1,0) \}
\]
(for $n > 1$),
and
\[
J = -\left( \frac{p^2}{n} + p \phi + \frac{n-1}{4} \phi^2 \right).
\]
If $n=1$ the hypersurface reduces to a circle and so $\tilde{\mathcal W} = \{ (0,0) \}$,
i.e., $\mathcal W$ is empty.

We may factorise $J$ as above where
\[
{\bf c} = \left( - \,\frac{1}{2}(n + \sqrt{n}), 1 \right) \;\; \mbox{\rm and} \;\;
\theta = -\left(\frac{p}{\sqrt{n}} + \frac{ \sqrt{n} -1}{2} \,\phi \right).
\]
So we have a factorisation of $J$ {\em over the rationals}
if and only if $n$ is a perfect square. This is exactly the
condition singled out by Painlev\'e analysis of the Bryant system in
\cite{BdP}.

\smallskip

We begin with a simple example.

\begin{example}
The Bryant system with $n=1$ takes a particularly simple form, as
we only have the zero vector in $\tilde{\mathcal W}$.

Now $J = -p (p + \phi)$, and we obtain a factorisation
with ${\bf c} = (-1, 1)$ or $(0,1)$. Choosing the latter
we have $\theta = p + \phi$.
We can now start the recursion with $F_{\bf c} = \theta$,
$\psi_{\bf c} =1$. Since
\[
\psi_{\bf c} + {\bf d}. \nabla F_{\bf c} = 1 + (1,-2)\cdot(1,1) =0
\]
we have a full solution to the recursion with all other terms zero.
The conserved quantity
is
\[
F = (p + \phi) e^u.
\]
In the variables of the discussion of the Bryant system in
\cite{Cetc} (chapter 1, \S 4.1) this is just $x+y$.
\end{example}

In \cite{DW1} (see \S 5), in certain situations with two weight vectors $v$
and $w$ we obtained some nontrivial conserved quantities as follows. We
were able to write $F_{\bf c}$ in two different ways
\[
F_{\bf c} = J \Gamma_v + \tau \,\theta^{s^\prime}
= J \Gamma_w + \rho \,\theta^{s}
\]
where
\[
({\bf v}+{\bf d})\cdot \nabla \tau = ({\bf w}+{\bf d})\cdot \nabla \rho =0,
\]
\[
({\bf v}+{\bf d})\cdot \nabla \theta = -\frac{1}{s^\prime}, \;\;\;
({\bf w} +{\bf d})\cdot \nabla \theta = -\frac{1}{s}, \;\;\;
\]
and $\Gamma_v, \Gamma_w$ are constant.
One now has a solution to the recursion with
\[
F_{{\bf c} +{\bf z} + {\bf d}} = - A_z \,\Gamma_z, \;\;\;
\psi_{{\bf c} +{\bf z} + {\bf d}}=0, \;\;\; {\rm with} \; {\bf z}={\bf v},{\bf w}
\]
and all other $F_{\bf a}, \psi_{\bf a}$ zero.

In our situation, we have $v=0$.
Taking
$\theta = -\left(\frac{p}{\sqrt{n}} + \frac{ \sqrt{n} -1}{2} \,\phi \right)$,
as above, we have
\[
{\bf d} \cdot \nabla \theta =-1, \;\;\;
({\bf d}+ {\bf w}) \cdot \nabla \theta = -1 + \frac{1}{\sqrt{n}}
\]
We see that the latter term is minus the reciprocal of
an integer if and only if $n=4$, when it is $-\frac{1}{2}$.

\begin{example}  \label{bryant-integral}
In the $n=4$ case we have
\[
{\bf d} = (4, -2) \;\; : \;\;
{\bf c} = (-3, 1), \;\; : \;\;
\theta = -\frac{1}{2} (p + \phi)
\]
We may take $\tau = (p + 2 \phi)$, so that ${\bf d} \cdot\nabla \tau =0$.
We now have
\[
 - \theta^2 = -J + \tau \theta = -\frac{1}{4}
(p +  \phi)^2
\]
so the above conditions are satisfied with $\rho = -1,
\Gamma_0 = -1, \Gamma_w = 0$
and $s=2, s^{\prime}=1$.

We obtain a nontrivial solution to the recursion relations:
\[
\psi_{\bf c} = \frac{1}{2}(p + \phi), \;\;\;
F_{\bf c} = \theta \psi_{\bf c} =  -\frac{1}{4}
(p + \phi)^2,
\]
\[
\psi_{{\bf c} + {\bf d} + {\bf w}} = 0, \;\;\;
F_{{\bf c} + {\bf d} + {\bf w}} = 0,
\]
\[
\psi_{{\bf c} + {\bf d}} = 0, \;\;\;
F_{{\bf c} + {\bf d}} = E.
\]
The conserved quantity is now
\begin{equation} \label{conserved1}
F = -\frac{1}{4}(p+\phi)^2 e^{-3q + u} + E e^{q-u}.
\end{equation}

We will show in \S 6 that this generalized first integral
allows us to write down the Bryant soliton in dimension 5 explicitly.
\end{example}

\begin{rmk} \label{Darboux}
It is interesting to observe that in the $n=4$ case one may also
find a conserved quantity by the method of Darboux polynomials (see \cite{PrS}
and \cite{Go}). It is convenient to use the variables of \cite{Cetc}
(chapter 1, \S 4.1), $x = \dot{h}, y = n \dot{h} - h \dot{u}$. Using as
our independent variable $T$ defined by $h\, dT = dt$, \cite{Cetc} derives the
equations
\begin{eqnarray*}
x^{\prime} &=& x^2 - xy + n-1 \\
y^{\prime} &=& x(y - nx).
\end{eqnarray*}
Given a planar system
\[
x^{\prime} = P(x,y), \;\;\; y^{\prime} = Q(x,y)
\]
representing the flow associated to the vector field
$X = P \partial_x + Q \partial_y$, we call a polynomial $J$
a {\em Darboux polynomial} if $X(J) = gJ$ for some polynomial $J$.
Sufficiently many Darboux polynomials may be used to construct
conserved quantities.

For our system we always have a Darboux polynomial $J_1 = nx^2 -y^2 + n(n-1)$,
for which $X(J_1) = 2x J_1$. This is associated to the soliton
conservation law, which in these variables is
\[
nx^2 -y^2 +n(n-1) = C h^2.
\]
If $n=4$ there is a second Darboux polynomial, $J_2 = 2x^2 -xy + 3$, which
satisfies $X(J_2) = (4x-y) J_2$. This means that $R = \frac{\sqrt{J_1}}{J_2}$
satisfies
\[
X(R) = (-3x + y) R = - {\rm div} (X) R.
\]
So $R$ is an integrating factor in the sense that $(RQ)_y + (RP)_x=0$.
There is now an $I$ such that $I_x = RQ$ and $I_y = -RP$, and $I$
is a conserved quantity.
\end{rmk}
\section{\bf Superpotentials}

A {\em superpotential} for a Hamiltonian is a time-independent
solution to the Hamilton-Jacobi equations, i.e.,  a function
${\bf q} \mapsto f({\bf q})$ on configuration space that solves the equation
\begin{equation}
H({\bf q}, df({\bf q})) =0.
\end{equation}
As we shall describe in the next section, superpotentials are important
because they define subsystems of the full Hamiltonian system, which
are often more tractable than the full system.
In \cite{DW2}, \cite{DW6}, and \cite{DW7} we analysed the existence of
superpotentials in the Einstein case.

In the steady soliton case the above equation becomes:
\[
J(\nabla f, \nabla f)= e^{{\bf d} \cdot {\bf q}} \left( E
       +  \sum_{w \in {\mathcal W}} \, A_w e^{{\bf w} \cdot {\bf q}} \right)
       \]
where $\nabla f$ denotes the Euclidean gradient with respect to the variables
$\bf q$.
We look for superpotentials of the form
\begin{equation} \label{supertype}
f = \sum \, f_{\bf c} \,e^{{\bf c} \cdot {\bf q}}
\end{equation}
in which $f_{\bf c}$ are constant and the sum ranges over a finite set ${\mathcal C}$ of points ${\bf c}$ in $\R^{r+1}$.
This leads to
\begin{equation} \label{supereq}
\sum_{{\bf a}+{\bf c} = {\bf  b}} J({\bf a},{\bf c})\, f_{\bf a}\,\, f_{\bf c} =\left\{ \begin{array}{ll}
 A_w & \mbox{if ${\bf b}={\bf d} + {\bf w}$ for some $w \in {\mathcal W}$} \\
 E  & \mbox{if ${\bf b}= {\bf d}$} \\
  0  & \mbox{otherwise}. \end{array} \right.
\end{equation}

The following lemma is often useful.
\begin{lemma} \label{vw}
Let $\bf{v}, \bf{w}$ be vectors in $\R^{r+1}$ with $v_{r+1} = w_{r+1} =0$.
Then
\begin{equation}
J({\bf v}+{\bf d}, {\bf w}+{\bf d}) = 1 -\sum_{i=1}^{r} \, \frac{v_i w_i}{d_i}.
\end{equation}
\end{lemma}
\begin{proof} By (\ref{bilnform}),
\begin{eqnarray*}
J({\bf d}+{\bf v}, \bf{d}+{\bf w}) &=& -\left( \sum_{i=1}^{r} \frac{(d_i + v_i)(d_i + w_i)}{d_i}
 - \sum_{i=1}^{r} (d_i + v_i) - \sum_{i=1}^{r} (d_i + w_i) + \left(\frac{n-1}{4}\right)
 (-2)^2 \right) \\
            &=& -\left( \sum_{i=1}^{r} \left(\frac{v_i w_i}{d_i} - d_i \right) + n-1 \right)\\
            &=& 1 - \sum_{i=1}^{r} \frac{v_i w_i}{d_i}.
\end{eqnarray*}
\end{proof}

We can deduce from (\ref{supereq}) the following result, whose
proof is essentially the same to that in the Einstein case
(see Propositions 2.2-2.6 of \cite{DW2}). We use ${\rm conv}(S)$ to
denote the convex hull of a set $S$. As mentioned earlier, we
will identify $\tilde{\mathcal W}$ with a set in $\R^{r+1}$ by associating
to each weight vector $w$ the extended vector ${\bf w} = (w,0)$
in $\R^{r+1}$.

\begin{lemma} \label{portmanteau}
$($i$)$ The convex hull of $\mathcal C$ contains the convex hull of
$\frac{1}{2} ({\bf d} +  \tilde{\mathcal W})$.

\smallskip 
$($ii$)$ If ${\bf a},{\bf c} \in {\mathcal C}$ and ${\bf
  a}+ {\bf c}$ cannot be written as the sum of two elements of
$\mathcal C$ distinct from ${\bf a},{\bf c}$, then either $J({\bf
  a},{\bf c})=0$ or ${\bf a}+{\bf c} \in {\bf d} + \tilde{\mathcal W}$. In particular,
if $\bf c$ is a vertex of conv$({\mathcal C})$, then either $\bf c$ is
$J$-null or $2{\bf c} = {\bf d}+{\bf w}$ for some $w \in \tilde{\mathcal W}$ and
$J({\bf c},{\bf c})\,f_{\bf c}^2 = A_w$.

\medskip
If we further assume that no vertex of conv$({\mathcal C})$ is $J$-null, then we
also have:

\medskip
$($iii$)$ The convex hull of $\mathcal C$ equals the convex hull of
$\frac{1}{2} ({\bf d} +  \tilde{\mathcal W})$. In particular
every element of $\mathcal C$ is of the form ${\bf c} = 
\frac{1}{2}({\bf d} + {\bf x})$
for ${\bf x} \in \R^{r+1}$ with $-1 \leq \sum x_i \leq 0$. Moreover,
$\sum x_i =-1$ if and only if $\bf x$ is a convex linear combination of nonzero
elements of $\tilde{\mathcal W}$ $($ i.e. of elements of $\mathcal W$ $)$.
Also $\sum x_i =0$ if and only if ${\bf x}=0$, that is, 
${\bf c} = \frac{1}{2} {\bf d}$.

\smallskip
$($iv$)$ If $\bf w$ is a vertex of conv$(\tilde{\mathcal W})$ 
then ${\bf w} +{\bf d} = 2 {\bf c}$
for some vertex $\bf c$ of conv$({\mathcal C})$. Moreover 
$J({\bf d} +{\bf w},{\bf d}+ {\bf w})$
has the same sign as $A_w$.    \flushright{\qed}
\end{lemma}

If we make the non-null vertex assumption in the second half of the above
Lemma, we can deduce a non-existence result for superpotentials, using
similar arguments to the result in the Einstein case for nonzero
cosmological constant (cf Theorem 10.1 of \cite{DW2}).

\begin{prop} \label{nonnull}
If no vertex of conv$({\mathcal C})$ is $J$-null, then there
are no superpotentials of the form $($\ref{supertype}$)$.
\end{prop}

\begin{proof}
Let $\bf w$ be a vertex of conv$({\mathcal W})$, so the line segment $0{\bf w}$
is an edge of conv$(\tilde{\mathcal W})$. The preceding lemma
\ref{portmanteau}
shows that ${\bf c}_0 = \frac{1}{2} {\bf d}$ and
${\bf c}_1 = \frac{1}{2} ({\bf d} + {\bf w})$ are vertices of conv$({\mathcal C})$
and ${\bf c}_0 {\bf c}_1$ is an edge of  conv$({\mathcal C})$.

Lemma \ref{vw}
shows that $J({\bf c}_0, {\bf c}_0) = \frac{1}{4} J({\bf d}, {\bf d}) = \frac{1}{4}$,
and also $J({\bf c}_0, {\bf c}_1) = \frac{1}{4} J({\bf d}, {\bf d} + {\bf w})
= \frac{1}{4}$. As $J({\bf c}_0, .)$ is an affine function on the 
edge ${\bf c}_0 {\bf c}_1$
it is therefore constant.

By considering the element of $\mathcal C$ on ${\bf c}_0 {\bf c}_1$ which is distinct
from ${\bf c}_0$ and closest to ${\bf c}_0$, and applying part (ii) of Lemma
\ref{portmanteau}, we obtain a contradiction.
\end{proof}

\begin{rmk}
It is interesting to note that in the Ricci-flat case there are several superpotentials
satisfying the non-null assumption (see the discussion
in \cite{DW2}, especially Theorem 6.1). So the above result is another manifestation
of the greater rigidity of the soliton equations compared to the Ricci-flat
Einstein equations.
\end{rmk}

\medskip
On the other hand, if we relax the non-null assumption,
we can find some examples of superpotentials in the steady case.

\begin{example} \label{bryant}
We first consider the situation of the Bryant soliton for which
$\tilde{\mathcal W} = \{ (0,0), (-1,0) \}$ if $n > 1$.

We observe using Lemma \ref{vw} that the only null vectors with zero second component
are $\pm(\sqrt{n},0)$. In particular if $n=4$ then we can take $\mathcal C$ to consist
of the vectors $\frac{1}{2} {\bf d}$ and $\frac{1}{2}({\bf d} + (-2,0))$, where ${\bf d} = (4, -2)$.
The first gives us ${\bf d} + (0,0)$ in the
expansion of $J(\nabla f, \nabla f)$, the cross term gives ${\bf d} + (-1,0)$,
and the last vector is null so does not contribute. We therefore obtain a superpotential
\begin{equation} \label{Bryantsuperpot}
f = 2 \sqrt{E}\, e^{2q_1 - u} + \frac{12}{\sqrt{E}}\,\,
e^{q_1 -u}
\end{equation}
for {\em positive} $E$.

When $n=1$, we  can obtain a superpotential in another way because the extended vectors
${\bf d} + (\pm 1, 0)$ are both $J$-null. We let $\mathcal C$  consist of the vectors
${\bf c}_1 = (1, -1)$ and ${\bf c}_2 = (0, -1)$. Note that the scalar curvature function is $0$, so
the parts of the superpotential equation involving ${\bf c}_i + {\bf c}_i$ automatically hold. The
single cross term leads to the condition $f_1 f_2 = E$. So if $E \neq 0$ we get the superpotential
$$ a\,e^{q_1 - u} + \frac{E}{a}\, e^{-u}, \,\,\, a \neq 0$$
while if $E=0$ we get for $a \neq 0$ we get the possibilities
$$f = a\, e^{q_1 - u},  \ \  \mbox{\rm and} \,\, f = a\, e^{-u}.$$
These will be referred to as the {\em limiting cases} of the previous superpotential.

\end{example}

\begin{example} \label{warped}
Next we consider the case of a warped product on two factors, so
that the metric is
\[
\bar{g}= dt^2 + h_1(t)^2 g_1 + h_2(t)^2 g_2
\]
where $ $($M_i, g_i$)$ $ is Einstein of dimension $d_i$ with positive Einstein
constant normalised to be $d_i(d_i - 1)$.

Now $r=2$ and $\tilde{\mathcal W} = \{ (0,0,0),(-1,0,0),(0,-1,0)\}$.
We pick ${\mathcal C}$ to consist of vectors
\begin{eqnarray*}
{\bf c}_1 &=& \frac{1}{2}({\bf d} + (-1,-1,0)), \\
{\bf c}_2 &=& \frac{1}{2}({\bf d} + (-1,1,0)), \\
{\bf c}_3 &=& \frac{1}{2}({\bf d} + (1,-1,0)).
\end{eqnarray*}
Now the cross terms give us the elements of $\tilde{\mathcal W}$, and $J({\bf c}_i,{\bf c}_j)$
is positive for $i \neq j$. If $E$ is positive, we may therefore
obtain a superpotential provided that ${\bf c}_1, {\bf c}_2, {\bf c}_3$ are null, which
translates into the condition $\frac{1}{d_1} + \frac{1}{d_2} =1$. Therefore, if
$(d_1,d_2)=(2,2)$, we obtain a  superpotential given by
$$ F = \sqrt{E}\, e^{\frac{1}{2}q_1 + \frac{3}{2}q_2 - u} + \sqrt{E}\, e^{\frac{3}{2}q_1 + \frac{1}{2} q_2 -u}
    +\frac{4}{\sqrt{E}}\,\, e^{\frac{1}{2} q_1 + \frac{1}{2} q_2 -u}.$$
\end{example}

\begin{rmk}
In the $n=4$ case of Example \ref{bryant}, and in Example \ref{warped}
the superpotential condition forces $E$ to be positive.
Recalling from the discussion preceding Proposition \ref{zeroE}
that in the steady case $E$ is just equal to minus the constant $C$, it
is interesting to observe that the condition $E > 0$ is also
a consequence of assuming completeness of the metric (cf. \cite{BDGW}).
\end{rmk}

\begin{example} \label{BBC}
This example involves the B\'erard-Bergery-Calabi ansatz, that is, we
take the hypersurface to be a circle bundle over a product of $r \geq 1$
Fano K\"ahler-Einstein manifolds of real dimension $d_i = 2m_i$. Thus
${\bf d}=(1, 2m_2, \cdots, 2m_r, -2).$

Now $\tilde{\mathcal W}$ consists of

(i) type III vectors with $1$ in the first place and $-2$ in the $i$th place
for $2 \leq i \leq r$,

(ii) type I vectors with $-1$ in $i$th place $(2 \leq i \leq r)$, and

(iii) the zero vector.

\smallskip
For $\mathcal C$, we take
\[
{\bf c}_1 = \frac{1}{2} ({\bf d} + (-1,0, \cdots,0)),
\]
${\bf c}_2, \ldots, {\bf c}_r$ to be $\frac{1}{2}({\bf d}+{\bf w})$
where ${\bf w}$ ranges over the the type III vectors, and
\[
{\bf c}_{r+1} = \frac{1}{2}({\bf d} + (1,0, \cdots, 0)).
\]
As $d_1 =1$, we see that ${\bf c}_1$ and ${\bf c}_{r+1}$ are null, while
${\bf c}_i : (2 \leq i \leq r)$ are mutually orthogonal, and also orthogonal
to ${\bf c}_{r+1}$. The cross terms ${\bf c}_1 + {\bf c}_i$ for $2 \leq i \leq r$
give the type I vectors in ${\bf d} + \tilde{\mathcal W}$,
while $2{\bf c}_i$ give the type III vectors and ${\bf c}_1 + {\bf c}_{r+1}$
gives the zero vector.

We let $A_i$ for $i=2, \ldots, r$ denote the constants $A_{\bf w}$
for the type III vectors, and $A_{r+i-1}$ denote the constants for the
corresponding type I vectors. We must take $f_i = \sqrt{-A_i d_i}$ for
$2 \leq i \leq r$ (possible since $A < 0$ for type III vectors).
We see finally that we obtain a superpotential
provided $A_{r+i-1}/\sqrt{-A_i d_i}$ is the same for each $2 \leq i \leq r$.

If we normalize the K\"ahler-Einstein metric on $M_i$ so that its K\"ahler class
equals $2\pi \alpha_i$ where $\alpha_i$ is the indivisible integral cohomology class
so that the first Chern class of $M_i$ is $\kappa_i \alpha_i$ with $\kappa_i > 0$,
then $A_{r+i-1} = d_i \kappa_i$. Also, suppose that the Euler class of the circle bundle
is $b_2 \alpha_2 + \cdots + b_r \alpha_r$. Then it follows that
$A_i = -\frac{1}{4}d_i b_i^2$ for $2 \leq i \leq r$. The conditions guaranteeing a
superpotential in the previous paragraph translate into requiring $\kappa_i/|b_i|$
to be independent of $i$. Note that this is precisely the condition under which
an explicit, complete, steady K\"ahler Ricci soliton was constructed in \cite{DW5}
(see Theorem 4.20(i) in that reference).

Note furthermore that unlike the previous examples, there is no constraint on the
sign of $E$ here.
\end{example}

\section{\bf Explicit Non-K\"ahler Steady Solitons}

One reason for the importance of superpotentials is that they give rise to a Lagrangian
section of the cotangent bundle of configuration space that is invariant under the
Hamiltonian flow. Using this section to pull back the Hamiltonian vector field to
configuration space, one obtains a first order subsystem of the canonical equations
(see \S 1 of \cite{DW2} for details). For the Hamiltonian system associated to the
cohomogeneity one gradient Ricci soliton equation, the first order subsystem is given by
\begin{equation} \label{subsystem}
 \dot{\bf q} = 2 {\bf v}^{-1} J \nabla f
\end{equation}
where ${\bf v}$ is the extended relative volume $e^{\frac{1}{2} {\bf d} \cdot {\bf q}} = v e^{-u}$
and $f$ is the superpotential.

In the following we shall consider the first order subsystems for Examples \ref{bryant}
and \ref{warped}, and show that they lead to explicit solutions of the soliton equation.
Explicit solutions arising from the superpotentials in Example \ref{BBC} were discussed
in \S 4 of \cite{DW5}.

\medskip

For Example \ref{bryant}  with $n=4$ we have
$$ J = - \left( \begin{array}{rr}
               \frac{1}{4}  &   \frac{1}{2} \\
               \frac{1}{2}  &   \frac{3}{4}  \end{array} \right), $$
so Eq. (\ref{subsystem}) and the superpotential constructed in that example yield
the system
\begin{eqnarray*}
\dot{q_1} &=& \frac{6}{\sqrt{E}} \,\,e^{-q_1}   \\
\dot{u}  &=& -\sqrt{E} + \frac{6}{\sqrt{E}} \,\,e^{-q_1}.
\end{eqnarray*}
Since $h(t)^2 = e^{q_1}$, we have $\frac{d}{dt}\, h^2 = \frac{6}{\sqrt{E}}$, and so
\begin{eqnarray*}
     h(t) &=& \frac{\sqrt{6}}{E^{1/4}} \sqrt{t + t_0} \\
     u(t) &=& -\sqrt{E}\, t + \log (t+t_0) + \,\mbox{\rm const}
\end{eqnarray*}
where $t_0$ is a constant.

Notice that there is no smooth soliton among this one-parameter family of solutions.
Assuming that we place the singular orbit at $t=0$, the smoothness conditions require
$h(0)=0$, so that $t_0 = 0$. But we also need $\dot{h}(0) = 1$, which is never
satisfied. However, the solitons are complete at infinity.

\begin{rmk} \label{conicalBryant}
The above family of explicit singular solitons was first constructed in
\cite{ACG} (see Proposition 2.2 and Remark 2.6). They belong to a family of such
solitons which occur in all dimensions $\geq 3$. For general dimensions the solitons
were constructed by dynamical systems methods. The existence of a superpotential help
to explain why dimension $5$ is special.
\end{rmk}

In the $n=1$ case of Example \ref{bryant} one easily checks that the first superpotential
gives rise to the first order system
\begin{eqnarray*}
\dot{q_1}  &=&  -a e^{\frac{1}{2}q_1} + \left(\frac{E}{a}\,\,\right) e^{-\frac{1}{2} q_1}  \\
\dot{u}  &=& -a e^{\frac{1}{2}q_1}.
\end{eqnarray*}

Note that the first equation above can be expressed in terms of $h$ as
\begin{equation} \label{spODE}
\dot{h} = \frac{a}{2} \left( \frac{E}{a^2} - h^2 \right),
\end{equation}
which is precisely of the form of Eq. (1.42) in Chapter I of \cite{Cetc}. The point is
that this equation follows immediately from the existence of a superpotential and
is not the consequence of ad hoc derivations. We will skip the detailed analysis of
the associated first order system since $2$-dimensional gradient Ricci solitons,
singular or otherwise, have been classified in detail in \cite{BeMe}. However, we do
want to point out the solutions of the first order system yield {\em all} the families
of steady solitons given in \cite{BeMe}.  In other words, all $2$-dimensional steady
gradient Ricci solitons, singular or not, have an associated conserved quantity that is
linear in momentum.

We also note that the $h$ identically constant case (necessarily with $E > 0$)
gives the cylinder, and the $1$-parameter family of cigar solitons (necessarily with $E>0$)
can be characterised as members of a $2$-parameter family of generically singular solitons
which also satisfy the smoothness condition $E=2a$.

The two limiting cases of the superpotential give rise respectively to singular solitons
which are the limits (as $E/a$ tends to $0$) of the family of ``exploding solitons", and
a family of flat cone solutions, which include Euclidean space as a special case.

\medskip

In Example \ref{warped}, we found a superpotential when $d_1 = d_2 = 2$. The
matrix for $J$ is therefore
$$ J = - \left(\begin{array}{rrr}
         \frac{1}{2}  &   0          &   \frac{1}{2}  \\
         0            & \frac{1}{2}  &  \frac{1}{2}   \\
         \frac{1}{2}  &  \frac{1}{2} &  \frac{3}{4}
         \end{array} \right).$$
Hence with $h_1^2= e^{q_1}$ and $h_2^2 = e^{q_2}$,  the associated first order subsystem is
\begin{eqnarray}
       \dot{q_1} &=& \frac{\sqrt{E}}{2}\, e^{-\frac{1}{2}q_1 + \frac{1}{2} q_2}
              - \frac{\sqrt{E}}{2}\, e^{\frac{1}{2}q_1 - \frac{1}{2} q_2}
              + \frac{2}{\sqrt{E}}\, e^{-\frac{1}{2}q_1 - \frac{1}{2} q_2}  \label{eq1} \\
       \dot{q_2} &=& -\frac{\sqrt{E}}{2}\, e^{-\frac{1}{2}q_1 + \frac{1}{2} q_2}
              + \frac{\sqrt{E}}{2}\, e^{\frac{1}{2}q_1 - \frac{1}{2} q_2}
              + \frac{2}{\sqrt{E}}\, e^{-\frac{1}{2}q_1 - \frac{1}{2} q_2} \label{eq2}  \\
       \dot{u} &=&    -\frac{\sqrt{E}}{2}\, e^{-\frac{1}{2}q_1 + \frac{1}{2} q_2}
              - \frac{\sqrt{E}}{2}\, e^{\frac{1}{2}q_1 - \frac{1}{2} q_2}
              + \frac{2}{\sqrt{E}}\, e^{-\frac{1}{2}q_1 - \frac{1}{2} q_2}.  \label{eq3}
\end{eqnarray}

The first two equations imply that
$$ \tr \,L = 2\left( \frac{\dot{h_1}}{h_1} + \frac{\dot{h_2}}{h_2}\right)
         = \frac{4}{\sqrt{E}}\, \frac{1}{h_1 h_2}.$$
So $h_1 h_2 = \frac{2}{\sqrt{E}} (t + t_0)$ for some constant $t_0$. Using this relation
in Eq. (\ref{eq1}), after some simplification, we obtain
$$ \frac{d}{dt}(h_1^2) = (t+t_0) - \frac{E}{4(t+t_0)}\, h_1^4 + \frac{h_1^2}{t+ t_0}.$$
Let us set $\beta:= \frac{h_1^2}{t+t_0}.$ Then $\beta$ satisfies
\begin{equation} \label{integral}
      4 \dot{\beta} = 4 - E \beta^2,
\end{equation}
which again is analogous to Eq. (\ref{spODE}).

A special solution of this equation is $\beta = \pm \frac{2}{\sqrt{E}}$, which gives
$$ h_1(t)^2 = \pm \frac{2}{\sqrt{E}} (t+ t_0).$$
Note that the corresponding metrics cannot be smooth.

If we take $\sqrt{E} \beta < 2$, we obtain
$$ \beta = \left(\frac{2}{\sqrt{E}}\right) \, \frac{e^{\sqrt{E}(t + t_1)} - 1}{e^{\sqrt{E}(t + t_1)} + 1},$$
where $t_1$ is a constant. This in turn yields
$$ h_1(t)^2 = \left(\frac{2}{\sqrt{E}}\right)\, (t+t_0)\,   \frac{e^{\sqrt{E}(t + t_1)} - 1}{e^{\sqrt{E}(t + t_1)} + 1},$$
$$ h_2(t)^2 = \left(\frac{2}{\sqrt{E}}\right)\, (t+t_0)\,   \frac{e^{\sqrt{E}(t + t_1)} + 1}{e^{\sqrt{E}(t + t_1)} - 1},$$
and
$$\dot{u} = -\sqrt{E} + \frac{1}{t+t_0} - \frac{2\sqrt{E}}{e^{2\sqrt{E}(t + t_1)} - 1}.$$

Let us now look for complete smooth solutions within this $3$-parameter family of solutions.
We need to have $h_1(0)= 0, \dot{h_1}(0) = 1, \dot{u}(0) = 0$ and $h_2(0) > 0, \dot{h_2}(0)=0.$
Hence $t_0=t_1 = 0$, and one easily checks that these choices imply that all smoothness
conditions at $t=0$ (the position of the zero section) hold. We therefore obtain the
following $1$-parameter family of explicit solutions

$$h_1(t) = \frac{\sqrt{2t}}{E^{1/4}}\, 
\left(\tanh \left( \frac{t \sqrt{E}}{2} \right) \right)^{1/2}, \,\,\,\,
    h_2(t) = \frac{\sqrt{2t}}{E^{1/4}}\, \left(\coth \left( \frac{t \sqrt{E}}{2} 
\right) \right)^{1/2}, $$

$$ u(t) = \log \left( \frac{t \sinh (\sqrt{E})}{\sinh (\sqrt{E} t)} \right) + \, \mbox{\rm const}.$$

It follows that the mean curvature of the hypersurfaces is $\tr \, L = \frac{2}{t}$ and the
scalar curvature of the soliton metric is
$$ \bar{R} = \frac{2 \sqrt{E}}{t}  - \frac{1}{t^2}
      - \frac{4 \sqrt{E}}{e^{2 \sqrt{E} t} -1} \left(\sqrt{E} - \frac{1}{t} \right)
       - \frac{4E}{(e^{2 \sqrt{E} t } -1)^2} .$$

Similarly, we can analyse solutions for which $2 < \sqrt{E} \beta$ holds. 
Now we have 
$$ h_1(t) = \frac{\sqrt{2}}{E^{1/4}} \sqrt{t+t_0} \left(\coth (\frac{\sqrt{E}}{2} (t+t_1)) \right)^{1/2}$$
and we obtain smooth complete solutions in which the $h_1, h_2$
in the case treated above are swapped.

\medskip
Finally, we will use the generalized first integral (\ref{conserved1}) we found
for the Bryant soliton system (Example \ref{bryant-integral}) in dimension 5
to obtain an explicit expression for the Bryant soliton in that dimension.

Using the Legendre transformation formulae of \S 3, we can rewrite (\ref{conserved1}) as
\begin{equation} \label{conserved2}
F = e^{-u} h^2 \left(E - \left( \frac{2 \dot{h}}{h} - \dot{u} \right)^2  \right).
\end{equation}
We note first that the conservation law $F=\mu$ can be rewritten in terms of a 
new variable $\beta := -\log (\frac{e^u}{h^2})$ as
\[
(\dot{\beta})^2 = E - \mu e^{-\beta}
\]
which can be integrated explicitly.
The Hamiltonian constraint becomes
\[
-(\dot{\beta})^2  - 4 \,\frac{\dot{\beta} \dot{h}}{h} + E + \frac{12}{h^2} =0.
\]
Setting $\alpha = h^2$ we have the linear equation in $\alpha$:
\[
2 \dot{\beta} \dot{\alpha} - \mu e^{-\beta} \alpha =12
\]
so the system has been integrated by quadratures. We recover the
soliton potential via
\[
u = -\beta + \log \alpha.
\]

In order to obtain the Bryant soliton, we assume $\mu < 0$ and choose
$$ \dot{\beta} = \sqrt{E - \mu e^{-\beta}}. $$
Upon integration we get
$$ \beta = - \log \left(\frac{E}{-\mu}\left( \coth^2(\frac{\sqrt{E}}{2} (t+ t_0)) -1 \right)   \right) 
$$
where $t_0$ is a constant. Substituting this into the first order equation for $\alpha$ results in
$$ \alpha = h(t)^2 = \frac{6}{\sqrt{E}}\left((t+t_0 + t_1) \coth(\frac{\sqrt{E}}{2} (t+ t_0) ) - \frac{2}{\sqrt{E}}  \right)$$
for some integration constant $t_1$. 

The smoothness conditions $h(0)=0$ and $\dot{h}(0) = 1$ are then easily seen to be
satisfied if we choose $t_0 = t_1 = 0$. The soliton potential is given (up to an 
additive constant) by
$$ u = \log \left( \left( \coth^2(\frac{\sqrt{E}}{2} t)  -1  \right) 
\left( t \coth(\frac{\sqrt{E}}{2} t) - \frac{2}{\sqrt{E}}  \right) \right).$$

\begin{rmk}  \label{singBryant}
$($a$)$ If we instead choose $t_0 = 0$ and $t_1 > 0$ in the above, we obtain a $1$-parameter
family of solutions which are complete at $\infty$ but which blow up like $\frac{1}{t}$
at the origin.

$($b$)$ If we choose $\mu$ to be $0$ in the conservation law above, we recover
the steady solitons with a conical singularity at $t=0$ associated with 
the superpotential (\ref{Bryantsuperpot}).

$($c$)$ We can also take $\mu$ to be positive. In this case we obtain a $1$-parameter
family of solutions which are complete at $\infty$ such that $h(0)=0$ and $\dot{h}(0) = +\infty$.
In fact 
$$ h(t)^2 = \alpha = \frac{6}{\sqrt{E}}\left((t+t_0) \tanh(\frac{\sqrt{E}}{2} (t+ t_0) ) - \frac{2}{\sqrt{E}}  \right)$$
where $t_0(E) > 0$ is the unique positive solution of 
$\frac{\sqrt{E}}{2} t_0 \tanh (\frac{\sqrt{E}}{2} t_0 ) = 1.$

$($d$)$ We certainly expect that singular solitons with the properties of (a) and (c)
above exist in all dimensions. As mentioned before (see Remark \ref{conicalBryant}) ones with
properties in (b) were found in \cite{ACG}.
\end{rmk}


\begin{thebibliography}{bbbbb}
\bibitem[ACG]{ACG} Alexakis, S., Chen, D. Z., and Fournodvalos, G.:{\em Singular
     Ricci solitons and their stability under the Ricci flow}, arXiv:1304.6449.
\bibitem[BeMe]{BeMe} Bernstein, J. and  Mettler, T.:{\em Two-dimensional gradient
       Ricci solitons revisited,} arXiv:1303.6854 (math.DG).
\bibitem[BdP]{BdP} Betancourt de la Parra, A.: {\em Painlev\'e analysis
         of the Bryant soliton}, arXiv:1310.7254 (math.DG).
\bibitem[BDGW]{BDGW} Buzano, M.,  Dancer, A.,  Gallaugher, M., and  Wang, M:
       {\em A Family of Steady Ricci Solitons and Ricci-flat Metrics}:
        arXiv:13089.6140 (math.DG)
\bibitem[Cho]{Cetc}  Chow, B., Chu, S. C.,  Glickenstein, D.,  Guenther, C.,
     Isenberg, J.,  Ivey, T.,  Knopf, D.,  Lu, P.,  Luo, F., and  Ni, L.:
    {\em The Ricci flow: techniques and applications, part I: geometric aspects},
     Mathematical Surveys and Monographs Vol. 135, American Math. Soc., (2007).
\bibitem[DHW]{DHW}  Dancer, A.,  Hall, S., and  Wang, M.: {\em Cohomogeneity one shrinking Ricci
    solitons: an analytic and numerical study,} Asian J. Math., {\bf 17}, (2013), no. 1, 33-61.
\bibitem[DW1]{DW1}  Dancer, A., and  Wang, M.: {\em The cohomogeneity one Einstein equations
      from the Hamiltonian viewpoint}, \ J. reine angew. Math., {\bf 524}, (2000), 97-128.
\bibitem[DW2]{DW2}  Dancer, A., and  Wang, M.:{\em Superpotentials for the cohomogeneity one
     Einstein equations}, Commun. Math. Phys., {\bf 260}, (2005), 75-115.
\bibitem[DW3]{DW3} Dancer, A., and  Wang, M.: {\em  New examples of non-K\"ahler
      Ricci solitons,} Math. Res. Lett., {\bf 16}, (2009), 349-363.
\bibitem[DW4]{DW4}  Dancer, A., and  Wang, M.: {\em  Non-K\"ahler expanding Ricci solitons,}
     Int. Math. Research Notices, {\bf 2009} no.6, (2009), 1107-1133.
\bibitem[DW5]{DW5}  Dancer, A., and  Wang, M.: {\em On Ricci solitons of cohomogeneity one},
    Ann. Glob. Anal. Geom., {\bf 39}, (2011), 259-292.
\bibitem[DW6]{DW6}  Dancer, A., and  Wang, M.: {\em Classification of superpotentials},
      Comm. Math. Phys., {\bf 284}, (2008), 583-647.
\bibitem[DW7]{DW7}  Dancer, A., and  Wang, M.: {\em Classifying superpotentials: three
      summands case,} J. Geom. Phys., {\bf 61}, (2011), 675-692.
\bibitem[Go]{Go} Goriely, A.: {\em Integrability and nonintegrability of dynamical systems,}
       Advanced Series in Nonlinear Dynamics {\bf 19}, World Scientific, (2001).
\bibitem[Ham]{Ham}  Hamilton, R. S.: {\em The formation of singularities in the Ricci
     flow}, Surveys in Differential Geometry, {\bf 2} (1995), 7-136.
\bibitem[Iv]{Iv}  Ivey, T.: {\em New examples of complete Ricci solitons},
       Proc. AMS, {\bf 122}, (1994), 241-245.
\bibitem[Pe]{Pe}  Perelman, G.: {\em The entropy formula for the Ricci flow and its
         geometric applications}, arXiv:math.DG/0211159.
\bibitem[PrS]{PrS} Prelle, M. J., and Singer, M. F.: {\em Elementary first integrals of
      differential equations,} Trans. A. M. S., {\bf 279}, (1983), 215-229.
\end{thebibliography}
\end{document}